\documentclass[draft,onecolumn]{IEEEtran}
\IEEEoverridecommandlockouts
\usepackage{cite}
\usepackage{amsmath,amssymb,amsfonts,amsthm}
\usepackage{algorithmic}
\usepackage{graphicx}
\usepackage{textcomp}
\usepackage{xcolor}
\def\BibTeX{{\rm B\kern-.05em{\sc i\kern-.025em b}\kern-.08em
    T\kern-.1667em\lower.7ex\hbox{E}\kern-.125emX}}
\newtheorem{thm}{Theorem}
\newtheorem{defn}{Definition}

\newtheorem{lem}{Lemma}
\newtheorem{pro}{Proposition}
\newtheorem{corollary}{Corollary}

\usepackage{geometry}
\usepackage{setspace}
\begin{document}

\title{  Uncertainty principles for the windowed Hankel transform}
\author{\IEEEauthorblockN{Wen-Biao, Gao$^{1}$,
		Bing-Zhao, Li$^{1,2,\star}$,}\\
	\IEEEauthorblockA{
		1. School of Mathematics and Statistics, Beijing Institute of Technology, Beijing 102488, P.R. China}\\
	2. Beijing Key Laboratory on MCAACI, Beijing Institute of Technology, Beijing 102488, P.R. China\\
$^{\star}$Corresponding author: li$\_$bingzhao@bit.edu.cn
}

\maketitle

\begin{abstract}
\label{abstract}
\begin{spacing}{1.25}
The aim of this paper is to prove some new uncertainty principles for the windowed Hankel transform. They include
uncertainty principle for orthonormal sequence, local uncertainty principle,
logarithmic uncertainty principle and Heisenberg-type uncertainty principle. As a side result, we obtain the Shapiro's dispersion theorem for the windowed Hankel transform.
\end{spacing}
\end{abstract}

\begin{IEEEkeywords}

Hankel transform, Windowed Hankel transform, Uncertainty principle
\end{IEEEkeywords}

\noindent  {\bf 1. Introduction}

The uncertainty principle is a very important theorem in harmonic analysis. It states that a nonzero function and its Fourier
transform cannot be simultaneously sharply concentrated.
As a classical uncertainty principle, the Heisenberg uncertainty principle has been extended to other transforms such as
the linear canonical transform [30], the fractional Fourier transform [27] and the Dunkl transform [28].
The uncertainty principle for the Hankel transform was first proved by Bowie [29], then other uncertainty relations for the Hankel transform
 have been investigated [1,3,4,28].

The Hankel transform is found as a very useful mathematical tool in many fields of physics, geophysics, signal processing and other fields [18,19,20,21].
In a series of papers [22-26] various kinds of Hankel transform have been discussed in details. Ghobber and Omri [8] introduced the windowed Hankel transform
and described its basic properties. Baccar et al. [5] discussed the time-frequency analysis of localization operators the windowed
transform in the Hankel setting.
In this paper, Our main aim is to introduce some uncertainty principles for the windowed Hankel transform.

To do so, we need to introduce some relevant contents.

The Hankel transform of order $\alpha$ is defined on $L^{1}_{\alpha}(\mathbb{R_{+}})$ by [15]
\begin{align}
		\begin{split}
		\mathcal{H}_{\alpha}(f)(\lambda)=\int^{+\infty}_{0}f(t)j_{\alpha}(\lambda t)\emph{d}\gamma_{\alpha}(t),
		\end{split}
	\end{align}
where $\alpha\geq-\frac{1}{2}$, $\emph{d}\gamma_{\alpha}$ the measure defined on $[0,+\infty)$ by $\emph{d}\gamma_{\alpha}(t)=t^{2\alpha+1}/2^{\alpha}\Gamma(\alpha+1)\emph{d}t$,
$\Gamma $ is the gamma function and
$j_{\alpha}$ is the Bessel function given by
\begin{equation*}
	j_{\alpha}(t)=\Gamma(\alpha+1)\sum^{\infty}_{n=0}\frac{(-1)^{n}}{n!\Gamma(n+\alpha+1)}\left(\frac{t}{2}\right)^{2n}
	\end{equation*}
The modified Bessel function $j_{\alpha}(t)$ has the following integral representation [9,11,12],
for every $t\in \mathbb{C}$, we have
\begin{align}
		\begin{split}
		j_{\alpha}(t)=\begin{cases}
		\frac{2\Gamma(\alpha+1)}{\sqrt{\pi}\Gamma(\alpha+1/2)}\int^{1}_{0}(1-x^{2})^{\alpha-1/2}\cos(tx)\emph{d}x,   & \textrm{if} \ \ \alpha>-\frac{1}{2}  \\
		\cos z,    &\textrm{if} \ \ \alpha=-\frac{1}{2}.
		\end{cases}
		\end{split}
	\end{align}
In particular for every $z\in\mathbb{R}$
\begin{align*}
		\begin{split}
		\mid j_{\alpha}(z)\mid\leq1
		\end{split}
	\end{align*}
The integral representation (2) shows that for each $n\in\mathbb{N}$ and $t\in \mathbb{C}$
\begin{align*}
		\begin{split}
		\mid j^{(n)}_{\alpha}(t)\mid\leq e^{|Im(t)|}
		\end{split}
	\end{align*}
In particular for each $n\in\mathbb{N}$ and $x\in \mathbb{R}$
\begin{align*}
		\begin{split}
		\mid j^{n}_{\alpha}(x)\mid\leq1
		\end{split}
	\end{align*}

$L^{p}_{\alpha}(\mathbb{R_{+}})$ the space of measurable function $f$ on $[0,+\infty)$ satisfying \cite{[3]}
\begin{align}
		\begin{split}
		\parallel f\parallel_{L^{p}_{\alpha}(\mathbb{R_{+}})}=\left( \int^{+\infty}_{0}|f(t)|^{p}\emph{d}\gamma_{\alpha}(t)\right)^{1/p}<+\infty, \ \ \ \textrm{if} \ p\in[1,+\infty),
		\end{split}
	\end{align}
\begin{equation}
		\begin{split}
		\parallel f\parallel_{L^{\infty}_{\alpha}(\mathbb{R_{+}})}=\textrm{ess} sup_{t\in[0,+\infty)}|f(t)|<+\infty, \ \ \ \textrm{if} \ p=+\infty,
		\end{split}
	\end{equation}
Moreover, the Hankel transform satisfies the following inversion formula and Parseval equality [15]:

(1) Inversion formula: Let $f \in L^{1}_{\alpha}(\mathbb{R_{+}})$ such that $\mathcal{H}_{\alpha}(f)\in L^{1}_{\alpha}(\mathbb{R_{+}})$, then for every $t\in [0,+\infty)$ we have
\begin{equation*}
		\begin{split}
		f(t)=\int^{+\infty}_{0}\mathcal{H}_{\alpha}(f)(r)j_{\alpha}(rt)\emph{d}\gamma_{\alpha}(r)
		\end{split}
	\end{equation*}

(2) Parseval equality: The Hankel transform can be extended to an isometric isomorphism form $L^{2}_{\alpha}(\mathbb{R_{+}})$
onto itself. Moreover for every $f,g \in L^{2}_{\alpha}(\mathbb{R_{+}})$  we have the following Parseval equality:
\begin{equation}
		\begin{split}
		\int^{+\infty}_{0}f(t)\overline{g(t)}\emph{d}\gamma_{\alpha}(t)=
\int^{+\infty}_{0}\mathcal{H}_{\alpha}(f)(r)\overline{\mathcal{H}_{\alpha}(f)(r)}\emph{d}\gamma_{\alpha}(r)
		\end{split}
	\end{equation}
For every measurable subset $E\subseteq[0,+\infty)$, $\chi_{E}$ is the characteristic function of $E$ such that \cite{[4]}
\begin{align}
		\begin{split}
			\chi_{E}(t)=\begin{cases}
		1,   &t\in E  \\
		0,    &t\in E^{c}.
		\end{cases}
		\end{split}
	\end{align}
The translation operator associated with the Hankel transform is defined by \cite{[2]}
\begin{equation}
		\begin{split}
		\tau^{\alpha}_{k}f(t)=\frac{\Gamma(\alpha+1)}{\Gamma(1/2)\Gamma(\alpha+(1/2))}
\int^{\pi}_{0}f(\sqrt{t^{2}+k^{2}+2tk\cos\theta})(\sin\theta)^{2\alpha}\emph{d}\theta
		\end{split}
	\end{equation}
The operator $\tau^{\alpha}_{k}$ can be also written by the formula
\begin{equation}
		\begin{split}
		\tau^{\alpha}_{k}f(t)=\int^{\infty}_{0}f(x)K(k,t,x)\emph{d}\gamma_{\alpha}(x)
		\end{split}
	\end{equation}
where $K(k,t,x)$ is the kernel given by
\begin{align*}
		\begin{split}
			K(k,t,x)=\begin{cases}
		\frac{2\pi^{\alpha+1/2}\Gamma(\alpha+1)^{2}}{\Gamma(\alpha+(1/2))}
\frac{\Delta(k,t,x)^{2\alpha-1}}{(ktx)^{2\alpha}},   &\textrm{if}\ \ |k-t|<x<k+t  \\
		0,    &\textrm{otherwise}.
		\end{cases}
		\end{split}
	\end{align*}
where
\begin{align*}
		\begin{split}
			\Delta(k,t,x)=((k+t)^{2}-x^{2})^{1/2}(x^{2}-(k-t)^{2})^{1/2}
		\end{split}
	\end{align*}
is the area of the triangle with side length $k,t,x$.

The kernel $K(k,t,x)$ is symmetric in the variables $k,t,x$, and satisfy
\begin{align*}
		\begin{split}
		\int^{+\infty}_{0}K(k,t,x)\emph{d}\gamma_{\alpha}(x)=1
		\end{split}
	\end{align*}
For every $f\in L^{p}_{\alpha}(\mathbb{R_{+}})$ and every $k\in \mathbb{R_{+}}$, the function $\tau^{\alpha}_{k}(f)$ belong to
the space $L^{p}_{\alpha}(\mathbb{R_{+}})$ and
\begin{equation}
		\begin{split}
		\|\tau^{\alpha}_{k}(f)\|_{L^{p}_{\alpha}(\mathbb{R_{+}})}\leq\|f\|_{L^{p}_{\alpha}(\mathbb{R_{+}})}
		\end{split}
	\end{equation}
In particular, for every $k,t\geq0$, we have
\begin{equation}
		\begin{split}
		\tau^{\alpha}_{k}(f)(t)=\tau^{\alpha}_{t}(f)(k)
		\end{split}
	\end{equation}
If $f\in L^{1}_{\alpha}(\mathbb{R_{+}})$, then
\begin{equation}
		\begin{split}
		\int^{+\infty}_{0}\tau^{\alpha}_{k}(f)(t)\emph{d}\gamma_{\alpha}(t)
=\int^{+\infty}_{0}f(t)\emph{d}\gamma_{\alpha}(t)
		\end{split}
	\end{equation}
The modulation operator associated with the Hankel transform is defined by
\begin{equation*}
		\begin{split}
		\mathcal{M}^{\alpha}_{s}g:=\mathcal{H}_{\alpha}(\sqrt{\tau^{\alpha}_{s}|\mathcal{H}_{\alpha}(g)|^{2}})
		\end{split}
	\end{equation*}
The convolution product associated with the Hankel transform is defined for two functions $f$ and $g$ by [3,6,7]
\begin{equation*}
		\begin{split}
		f\sharp_{\alpha}g(t)=\int^{+\infty}_{0}f(r)\tau^{\alpha}_{t}(g)(r)\emph{d}\gamma_{\alpha}(r)
		\end{split}
	\end{equation*}
and
\begin{equation}
		\begin{split}
		\mathcal{H}_{\alpha}(f\sharp_{\alpha}g)=\mathcal{H}_{\alpha}(f)\mathcal{H}_{\alpha}(g)
		\end{split}
	\end{equation}

The paper is organized as follows: In Section 2, we present some preliminaries related to the windowed Hankel transform.
 Some different uncertainty principles associated with the windowed Hankel transform are provided in Section 3.

\noindent  {\bf 2. The windowed Hankel transform}

Let $g\in L^{2}_{\alpha}(\mathbb{R_{+}})$ and $s\in \mathbb{R}_{+}$, the modulation of $g$ is defined by [5,8]
\begin{equation}
		\begin{split}
		\mathcal{M}^{\alpha}_{s}g=\mathcal{H}_{\alpha}(\sqrt{\tau^{\alpha}_{s}|\mathcal{H}_{\alpha}(g)|^{2}})
		\end{split}
	\end{equation}
Then for every $g\in L^{2}_{\alpha}(\mathbb{R_{+}})$ and $s\in \mathbb{R}_{+}$, we obtained
\begin{equation}
		\begin{split}
		\|\mathcal{M}^{\alpha}_{s}g\|_{L^{2}_{\alpha}(\mathbb{R_{+}})}= \|g\|_{L^{2}_{\alpha}(\mathbb{R_{+}})}
		\end{split}
	\end{equation}
For a non-zero window function $g\in L^{2}_{\alpha}(\mathbb{R_{+}})$ and $(k,s)\in \mathbb{R}_{+}^{2}$, the function $g^{\alpha}_{k,s}$ defined by
\begin{equation}
		\begin{split}
		g^{\alpha}_{k,s}=\tau^{\alpha}_{k}\mathcal{M}^{\alpha}_{s}g
		\end{split}
	\end{equation}
Now, for any function $f\in L^{2}_{\alpha}(\mathbb{R_{+}})$, we define the windowed Hankel transform by [5,8]
\begin{equation}
		\begin{split}
		W^{\alpha}_{g}(f)(k,s)=\int^{+\infty}_{0}f(t)\overline{g^{\alpha}_{k,s}(t)}\emph{d}\gamma_{\alpha}(t), \ \ \ (k,s)\in \mathbb{R}_{+}^{2}
		\end{split}
	\end{equation}
which can be also written in the form
\begin{equation}
		\begin{split}
		W^{\alpha}_{g}(f)(k,s)=f\sharp_{\alpha}\overline{\mathcal{M}^{\alpha}_{s}g(k)}
		\end{split}
	\end{equation}
Define the measure $\nu_{\alpha}$ on $\mathbb{R^{*}_{+}}\times\mathbb{R_{+}}$, by
\begin{equation*}
		\begin{split}
		\emph{d}\nu_{\alpha}(k,s)=\emph{d}\gamma_{\alpha}(k)\emph{d}\gamma_{\alpha}(s)
		\end{split}
	\end{equation*}
The windowed Hankel transform satisfies the following properties [5,8].
\begin{pro}
Let $g\in L^{2}_{\alpha}(\mathbb{R_{+}})$ be a non-zero window function, then we have

(1) The Cauchy-Schwarz inequality: For any $f\in L^{2}_{\alpha}(\mathbb{R_{+}})$,
\begin{equation}
		\begin{split}
		\|W^{\alpha}_{g}(f)(k,s)\|_{L^{\infty}_{\alpha}(\mathbb{R_{+}}\times \mathbb{R_{+}^{*}})}\leq \|f\|_{L^{2}_{\alpha}(\mathbb{R_{+}})}\|g\|_{L^{2}_{\alpha}(\mathbb{R_{+}})}
		\end{split}
	\end{equation}
(2) Plancherel's formula: For any $f\in L^{2}_{\alpha}(\mathbb{R_{+}})$,
\begin{equation}
		\begin{split}
		\|W^{\alpha}_{g}(f)(k,s)\|_{L^{2}_{\alpha}(\mathbb{R_{+}}\times \mathbb{R_{+}^{*}})}=
\|f\|_{L^{2}_{\alpha}(\mathbb{R_{+}})}\|g\|_{L^{2}_{\alpha}(\mathbb{R_{+}})}
		\end{split}
	\end{equation}
(3) Orthogonality relation: For any $f,h\in L^{2}_{\alpha}(\mathbb{R_{+}})$,
\begin{equation}
		\begin{split}
		\int^{+\infty}_{0}\int^{+\infty}_{0}W^{\alpha}_{g}(f)(k,s)\overline{W^{\alpha}_{g}(h)(k,s)}
\emph{d}\nu_{\alpha}(k,s)=\|g\|^{2}_{L^{2}_{\alpha}
(\mathbb{R_{+}})}\int^{+\infty}_{0}f(t)\overline{h(t)}\emph{d}\gamma_{\alpha}(t)
		\end{split}
	\end{equation}
(4) Reproducing kernel Hilbert space: For any $(k',s');(k,s)\in \mathbb{R_{+}}\times \mathbb{R_{+}^{*}}$,
\begin{equation}
		\begin{split}
		H_{g}(k',s';k,s)=\frac{1}{\|g\|^{2}_{L^{2}_{\alpha}(\mathbb{R_{+}})}}
g^{\alpha}_{k',s'}\sharp_{\alpha}\overline{\mathcal{M}^{\alpha}_{s}g(k)}=
\frac{1}{\|g\|^{2}_{L^{2}_{\alpha}(\mathbb{R_{+}})}}W^{\alpha}_{g}(g^{\alpha}_{k',s'})(k,s)
		\end{split}
	\end{equation}
Furthermore, the kernel is pointwise bounded
\begin{equation}
		\begin{split}
		\mid H_{g}(k',s';k,s)\mid \leq1, \ \ \forall (k',s');(k,s)\in \mathbb{R_{+}}\times \mathbb{R_{+}^{*}}
		\end{split}
	\end{equation}
\end{pro}

\noindent{\bf 3. Uncertainty principle for the windowed Hankel transform}

In this section we obtain some uncertainty principles for the windowed Hankel transform.

We consider the following orthogonal projections [13,14]:

(1) $O_{g}$ is the orthogonal projection from $ L^{2}_{\alpha}(\mathbb{R_{+}}\times \mathbb{R_{+}^{*}})$ onto $W^{\alpha}_{g}(L^{2}_{\alpha}(\mathbb{R_{+}})$,
$R_{\phi}$ is its range.

(2)  $O_{E}$ is the orthogonal projection on $ L^{2}_{\alpha}(\mathbb{R_{+}}\times \mathbb{R_{+}^{*}})$ defined by
\begin{equation}
		\begin{split}
O_{E}F=\chi_{E}F, \ \ F\in L^{2}_{\alpha}(\mathbb{R_{+}}\times \mathbb{R_{+}^{*}})
		\end{split}
	\end{equation}
where $E\subset \mathbb{R_{+}}$, $R_{E}$ is its range.

We define $\|O_{E}O_{g}\|=sup\{\|O_{E}O_{g}(F)\|_{L^{2}_{\alpha}(\mathbb{R_{+}}\times \mathbb{R_{+}^{*}})}, \ F\in L^{2}_{\alpha}(\mathbb{R_{+}}\times \mathbb{R_{+}^{*}});
\|F\|_{L^{2}_{\alpha}(\mathbb{R_{+}}\times \mathbb{R_{+}^{*}})}=1\}$.

\begin{thm}
Let $g\in L^{2}_{\alpha}(\mathbb{R_{+}})$ be a non-zero window function. For any subset
$E\subset \mathbb{R_{+}}\times \mathbb{R_{+}^{*}}$ and finite measure $\nu_{\alpha}(E)<+\infty$, then $O_{E}O_{g}$ is a Hilbert-Schmid
operator and we have the following estimation:
\begin{equation}
		\begin{split}
\|O_{E}O_{g}\|^{2}\leq\nu_{\alpha}(E)
		\end{split}
	\end{equation}
\end{thm}
\begin{proof}
According to the paper \cite{[10]}, for every function $F\in L^{2}_{\alpha}(\mathbb{R_{+}}\times \mathbb{R_{+}^{*}})$, the orthogonal projection $O_{g}$
can be expressed as:
\begin{equation*}
		\begin{split}
O_{g}(F)(k,s)=\int^{+\infty}_{0}\int^{+\infty}_{0}F(k',s')H_{g}(k',s';k,s)\emph{d}\nu_{\alpha}(k',s')
		\end{split}
	\end{equation*}
where $H_{g}(k',s';k,s)$ is defined by (21). By (23), we get
\begin{equation*}
		\begin{split}
O_{E}O_{g}(F)(k,s)=\int^{+\infty}_{0}\int^{+\infty}_{0}\chi_{E}(k,s)F(k',s')H_{g}(k',s';k,s)\emph{d}\nu_{\alpha}(k',s')
		\end{split}
	\end{equation*}
Using (9), (15), (20) and Fubini's theorem, we have
\begin{equation}
		\begin{split}
\|O_{E}O_{g}\|^{2}_{HS}&=\int^{+\infty}_{0}\int^{+\infty}_{0}\int^{+\infty}_{0}\int^{+\infty}_{0}
|\chi_{E}(k,s)|^{2}|H_{g}(k',s';k,s)|^{2}\emph{d}\nu_{\alpha}(k',s')\emph{d}\nu_{\alpha}(k,s)\\
&=\int^{+\infty}_{0}\int^{+\infty}_{0}|\chi_{E}(k,s)|^{2}\left(\int^{+\infty}_{0}\int^{+\infty}_{0}
\left|\frac{1}{\|g\|^{2}_{L^{2}_{\alpha}(\mathbb{R_{+}})}}W^{\alpha}_{g}(g^{\alpha}_{k',s'})(k,s)\right|^{2}
\emph{d}\nu_{\alpha}(k',s')\right)\emph{d}\nu_{\alpha}(k,s)\\
&=\frac{1}{\|g\|^{2}_{L^{2}_{\alpha}(\mathbb{R_{+}})}}\int_{E}\int_{E}\left(\int^{+\infty}_{0}\int^{+\infty}_{0}
\frac{1}{\|g\|^{2}_{L^{2}_{\alpha}(\mathbb{R_{+}})}}\left|W^{\alpha}_{g}(g^{\alpha}_{k,s})(k',s')\right|^{2}
\emph{d}\nu_{\alpha}(k',s')\right)\emph{d}\nu_{\alpha}(k,s)\\
&\leq\frac{\|g\|^{2}_{L^{2}_{\alpha}(\mathbb{R_{+}})}}{\|g\|^{2}_{L^{2}_{\alpha}(\mathbb{R_{+}})}}\nu_{\alpha}(E)
=\nu_{\alpha}(E)
		\end{split}
	\end{equation}
According to $O_{E}O_{g}$ is an integral operator with Hilbert-Schmidt kernel, we have
$\|O_{E}O_{g}\|^{2}\leq\|O_{E}O_{g}\|^{2}_{HS}$. Which completes the proof.
\end{proof}
\begin{thm} (Uncertainty principle for orthonormal sequence)
Let $g\in L^{2}_{\alpha}(\mathbb{R_{+}})$ be a window function, $(\varphi_{n})_{n\in\mathbb{N}}$ be an orthonormal sequence in
 $L^{2}_{\alpha}(\mathbb{R_{+}})$ and any subset $E\subset \mathbb{R_{+}}\times \mathbb{R_{+}^{*}}$. If $\nu_{\alpha}(E)<+\infty$, then for every
 $N\in \mathbb{N^{*}}$, we have
\begin{equation}
		\begin{split}
\sum^{N}_{n=1}(1-\|\chi_{E^{c}}W^{\alpha}_{g}(\varphi_{n})\|_{L^{2}_{\alpha}(\mathbb{R_{+}}\times \mathbb{R_{+}^{*}})})
\leq \nu_{\alpha}(E)
		\end{split}
	\end{equation}
\end{thm}
\begin{proof}
Let $(e_{n})_{n\in\mathbb{N}}$ be an orthonormal basis of $L^{2}_{\alpha}(\mathbb{R_{+}})$, since $O_{E}O_{g}$ is an integral operator with Hilbert-Schmidt kernel and satisfy relation (25), hence the positive operator $O_{g}O_{E}O_{g}$ satisfies
\begin{equation*}
		\begin{split}
\sum_{n\in\mathbb{N}}\langle O_{g}O_{E}O_{g}e_{n},e_{n}\rangle_{L^{2}_{\alpha}(\mathbb{R_{+}}\times \mathbb{R_{+}^{*}})}=\|O_{E}O_{g}\|^{2}_{HS}\leq\nu_{\alpha}(E)<+\infty
		\end{split}
	\end{equation*}
where $\langle\cdot,\cdot\rangle$ is the inner product in $L^{2}_{\alpha}(\mathbb{R_{+}}\times \mathbb{R_{+}^{*}})$.

According to the paper \cite{[17]}, the positive operator $O_{g}O_{E}O_{g}$ is a trace class operator and
\begin{equation*}
		\begin{split}
Tr(O_{g}O_{E}O_{g})=\|O_{E}O_{g}\|^{2}_{HS}\leq\nu_{\alpha}(E)
		\end{split}
	\end{equation*}
Since $(\varphi_{n})_{n\in\mathbb{N}}$ be an orthonormal sequence in $L^{2}_{\alpha}(\mathbb{R_{+}})$, then by (20) we deduce
that $(W^{\alpha}_{g}(\varphi_{n}))_{n\in\mathbb{N}}$ is an orthonormal sequence in $L^{2}_{\alpha}(\mathbb{R_{+}}\times \mathbb{R_{+}^{*}})$, thus
\begin{equation*}
		\begin{split}
\sum^{N}_{n=1}\langle O_{E}W^{\alpha}_{g}(\varphi_{n}),W^{\alpha}_{g}(\varphi_{n})\rangle _{L^{2}_{\alpha}(\mathbb{R_{+}}\times \mathbb{R_{+}^{*}})}=
\sum^{N}_{n=1}\langle O_{g}O_{E}O_{g}W^{\alpha}_{g}(\varphi_{n}),W^{\alpha}_{g}(\varphi_{n})\rangle _{L^{2}_{\alpha}(\mathbb{R_{+}}\times \mathbb{R_{+}^{*}})}
\leq Tr(O_{g}O_{E}O_{g})
		\end{split}
	\end{equation*}
Hence we obtain
\begin{equation*}
		\begin{split}
\sum^{N}_{n=1}\langle O_{E}W^{\alpha}_{g}(\varphi_{n}),W^{\alpha}_{g}(\varphi_{n})\rangle _{L^{2}_{\alpha}(\mathbb{R_{+}}\times \mathbb{R_{+}^{*}})}
\leq \nu_{\alpha}(E)
		\end{split}
	\end{equation*}
On the other hand, by Cauchy-Schwarz inequality, we have for every $n, 1\leq n\leq N$,
\begin{equation*}
		\begin{split}
\langle O_{E}W^{\alpha}_{g}(\varphi_{n}),W^{\alpha}_{g}(\varphi_{n})\rangle _{L^{2}_{\alpha}(\mathbb{R_{+}}\times \mathbb{R_{+}^{*}})}=
1-\langle O_{E^{c}}W^{\alpha}_{g}(\varphi_{n}),W^{\alpha}_{g}(\varphi_{n})\rangle _{L^{2}_{\alpha}(\mathbb{R_{+}}\times \mathbb{R_{+}^{*}})}
\geq1-\|\chi_{E^{c}}W^{\alpha}_{g}(\varphi_{n})\|_{L^{2}_{\alpha}(\mathbb{R_{+}}\times \mathbb{R_{+}^{*}})}
		\end{split}
	\end{equation*}
Thus, we obtain
\begin{equation*}
		\begin{split}
\sum^{N}_{n=1}(1-\|\chi_{E^{c}}W^{\alpha}_{g}(\varphi_{n})\|_{L^{2}_{\alpha}(\mathbb{R_{+}}\times \mathbb{R_{+}^{*}})})\leq
\sum^{N}_{n=1}\langle O_{E}W^{\alpha}_{g}(\varphi_{n}),W^{\alpha}_{g}(\varphi_{n})\rangle _{L^{2}_{\alpha}(\mathbb{R_{+}}\times \mathbb{R_{+}^{*}})}
\leq \nu_{\alpha}(E)
		\end{split}
	\end{equation*}
\end{proof}
\begin{defn}
Let $\delta>0$ and $E\subset\mathbb{R^{*}_{+}\times\mathbb{R_{+}}}$ be a measurable subset. Let
$f, g\in L^{2}_{\alpha}(\mathbb{R_{+}})$ be two non-zero functions. We say that $W^{\alpha}_{g}(f)$
is $\delta$-time-frequency concentrated on $E$, if
\begin{equation*}
		\begin{split}
\|\chi_{E^{c}}W^{\alpha}_{g}(f)\|_{L^{2}_{\alpha}(\mathbb{R_{+}}\times \mathbb{R_{+}^{*}})}\leq\delta
\|W^{\alpha}_{g}(f)\|_{L^{2}_{\alpha}(\mathbb{R_{+}}\times \mathbb{R_{+}^{*}})}
		\end{split}
	\end{equation*}
\end{defn}
According to Theorem 2, we obtain the following result.
\begin{pro}
Let $0<\delta<1, r>0$ and $B_{r}=\{(k,s)\in \mathbb{R^{*}_{+}\times\mathbb{R_{+}}}\mid \ \mid(k,s)\mid\leq r\}$.
Let $g\in L^{2}_{\alpha}(\mathbb{R_{+}})$ be a window function, $(\varphi_{n})_{1\leq n\leq N}$ be an orthonormal sequence in $L^{2}_{\alpha}(\mathbb{R_{+}})$. If $W^{\alpha}_{g}(\varphi_{n})$ is $\delta$-time-frequency concentrated in the ball $B_{r}$ for every $1\leq n\leq N$, then
\begin{equation}
		\begin{split}
N
\leq \frac{r^{4(\alpha+1)}}{2^{2(\alpha+1)}\Gamma(2\alpha+3)(1-\delta)}
		\end{split}
	\end{equation}

\end{pro}
\begin{proof}
Using Theorem 2, we have
\begin{equation}
		\begin{split}
\sum^{N}_{n=1}(1-\|\chi_{B_{r}^{c}}W^{\alpha}_{g}(\varphi_{n})\|_{L^{2}_{\alpha}(\mathbb{R_{+}}\times \mathbb{R_{+}^{*}})})
\leq \nu_{\alpha}(B_{r})
		\end{split}
	\end{equation}
For every $1\leq n\leq N$, we obtain
\begin{equation*}
		\begin{split}
\|\chi_{B_{r}^{c}}W^{\alpha}_{g}(f)\|_{L^{2}_{\alpha}(\mathbb{R_{+}}\times \mathbb{R_{+}^{*}})}\leq\delta
		\end{split}
	\end{equation*}
and
\begin{equation*}
		\begin{split}
\nu_{\alpha}(B_{r})=\frac{r^{4(\alpha+1)}}{2^{2(\alpha+1)}\Gamma(2\alpha+3)}
		\end{split}
	\end{equation*}
Hence
\begin{equation*}
		\begin{split}
N(1-\delta)
\leq \frac{r^{4(\alpha+1)}}{2^{2(\alpha+1)}\Gamma(2\alpha+3)}
		\end{split}
	\end{equation*}
\end{proof}
\begin{defn}
Let $g\in L^{2}_{\alpha}(\mathbb{R_{+}})$ be a window function and $f\in L^{2}_{\alpha}(\mathbb{R_{+}})$, we define
the generalized $p$th time-frequency dispersion of the windowed Hankel transform by
\begin{equation*}
		\begin{split}
\rho_{p}(W^{\alpha}_{g}(f))=\left(\int^{+\infty}_{0}\int^{+\infty}_{0}
\mid(k,s)\mid^{p}|W^{\alpha}_{g}(f)(k,s)|^{2}\emph{d}\nu_{\alpha}(k,s)\right)^{1/p}
		\end{split}
	\end{equation*}
where $p>0$ and $\mid(k,s)\mid=\sqrt{k^{2}+s^{2}}$.

The time dispersion of the windowed Hankel transform is defined by
\begin{equation*}
		\begin{split}
\rho_{k,p}(W^{\alpha}_{g}(f))=\left(\int^{+\infty}_{0}\int^{+\infty}_{0}
\mid k\mid^{p}|W^{\alpha}_{g}(f)(k,s)|^{2}\emph{d}\nu_{\alpha}(k,s)\right)^{1/p}
		\end{split}
	\end{equation*}

The frequency dispersion of the windowed Hankel transform is defined by
\begin{equation*}
		\begin{split}
\rho_{s,p}(W^{\alpha}_{g}(f))=\left(\int^{+\infty}_{0}\int^{+\infty}_{0}
\mid s\mid^{p}|W^{\alpha}_{g}(f)(k,s)|^{2}\emph{d}\nu_{\alpha}(k,s)\right)^{1/p}
		\end{split}
	\end{equation*}
\end{defn}
\begin{corollary}
Let $g\in L^{2}_{\alpha}(\mathbb{R_{+}})$ be a window function,
$(\varphi_{n})_{1\leq n\leq N}$ be an orthonormal sequence in $L^{2}_{\alpha}(\mathbb{R_{+}})$ and $p>0$.
Fix $Y>0$, if the sequence $\rho_{p}(W^{\alpha}_{g}(\varphi_{n}))\leq Y$ for every $1\leq n\leq N$, then
\begin{equation*}
		\begin{split}
N
\leq \frac{2^{\frac{8}{p}(\alpha+1)-2\alpha-1}Y^{4(\alpha+1)}}{\Gamma(2\alpha+3)}
		\end{split}
	\end{equation*}
\end{corollary}
\begin{proof}
Since, for every $r>0$, we have
\begin{equation*}
		\begin{split}
\|\chi_{B_{r}^{c}}W^{\alpha}_{g}(\varphi_{n})\|^{2}_{L^{2}_{\alpha}(\mathbb{R_{+}}\times \mathbb{R_{+}^{*}})}&=
\int_{\mid(k,s)\mid\geq r}\int_{|(k,s)|\geq r}|(k,s)|^{-p}|(k,s)|^{p}|W^{\alpha}_{g}(\varphi_{n})(k,s)|^{2}\emph{d}\nu_{\alpha}(k,s)\\
&\leq r^{-p}\rho^{p}_{p}(W^{\alpha}_{g}(\varphi_{n}))\leq r^{-p}Y^{p}
		\end{split}
	\end{equation*}
if $r=4^{1/p}Y$, we have
\begin{equation*}
		\begin{split}
\|\chi_{B_{(4^{1/p}Y)}^{c}}W^{\alpha}_{g}(\varphi_{n})\|^{2}_{L^{2}_{\alpha}(\mathbb{R_{+}}\times \mathbb{R_{+}^{*}})}\leq \frac{1}{4}
		\end{split}
	\end{equation*}

Hence for every $1\leq n\leq N$, $W^{\alpha}_{g}(\varphi_{n})$ is $\frac{1}{2}$-concentrated in the ball $B_{4^{1/p}Y}$, applying Proposition 2,
we obtain
\begin{equation*}
		\begin{split}
N
\leq \frac{2^{\frac{8}{p}(\alpha+1)-2\alpha-1}Y^{4(\alpha+1)}}{\Gamma(2\alpha+3)}
		\end{split}
	\end{equation*}
\end{proof}
\begin{lem}
Let $g\in L^{2}_{\alpha}(\mathbb{R_{+}})$ be a window function,
$(\varphi_{n})_{1\leq n\leq N}$ be an orthonormal sequence in $L^{2}_{\alpha}(\mathbb{R_{+}})$ and $p>0$.
then there exits $i_{0}\in\mathbb{Z}$ such that
\begin{equation}
		\begin{split}
\rho_{p}(W^{\alpha}_{g}(\varphi_{n}))\geq2^{i_{0}}. \ \ \ \forall n\in\mathbb{N}
		\end{split}
	\end{equation}
\end{lem}
\begin{proof}
For each $i\in\mathbb{Z}$, we define
\begin{equation*}
		\begin{split}
P_{i}=\{n\in\mathbb{N}|\rho_{p}(W^{\alpha}_{g}(\varphi_{n}))\in[2^{i-1},2^{i})\}
		\end{split}
	\end{equation*}
then for every $n\in P_{i}$, we have $\rho_{p}(W^{\alpha}_{g}(\varphi_{n}))\leq2^{i}$.

Hence
\begin{equation*}
		\begin{split}
\int_{\mid(k,s)\mid\geq 2^{i+2/p}}\int_{|(k,s)|\geq 2^{i+2/p}}|W^{\alpha}_{g}(\varphi_{n})(k,s)|^{2}\emph{d}\nu_{\alpha}(k,s)\leq\frac{1}{4}
\end{split}
	\end{equation*}

For every $n\in P_{i}$, $W^{\alpha}_{g}(\varphi_{n})$ is $\frac{1}{2}$-concentrated in the ball $B_{2^{i+2/p}}$,
according to Proposition 2, we have $P_{i}$ is finite and
\begin{equation}
		\begin{split}
N_{i}
\leq \frac{2^{\frac{8}{p}(\alpha+1)-2\alpha-1}}{\Gamma(2\alpha+3)}2^{4(\alpha+1)i}
		\end{split}
	\end{equation}
where $N_{i}$ is the number of elements in $P_{i}$. We see that $P_{i}$ is empty for all $i<i_{0}$, so
$\rho_{p}(W^{\alpha}_{g}(\varphi_{n}))\geq2^{i_{0}}$.
\end{proof}
\begin{thm} (Shapiro's dispersion theorem for the windowed Hankel
transform)
Let $g\in L^{2}_{\alpha}(\mathbb{R_{+}})$ be a window function and
$(\varphi_{n})_{1\leq n\leq N}$ be an orthonormal sequence in $L^{2}_{\alpha}(\mathbb{R_{+}})$, then for
 $p>0$ and $N\in\mathbb{N}$, we have
\begin{equation*}
		\begin{split}
\sum^{N}_{n=1}\rho^{p}_{p}(W^{\alpha}_{g}(\varphi_{n}))\geq
N^{1+\frac{p}{4(\alpha+1)}}\left( \frac{3\Gamma(2\alpha+3)}{2^{\frac{8}{p}(\alpha+1)+6\alpha+8}}\right)^{\frac{p}{4(\alpha+1)}}
		\end{split}
	\end{equation*}
\end{thm}
\begin{proof}
Let $m\in\mathbb{Z}$, then according to (28), for every $m>i_{0}$, the number of element in $\bigcup_{i=i_{0}}^{m}P_{i}$ is less than
$Q2^{4(\alpha+1)m}$, where
\begin{equation*}
		\begin{split}
Q=\frac{2^{\frac{8}{p}(\alpha+1)+2\alpha+3}}{3\Gamma(2\alpha+3)}
		\end{split}
	\end{equation*}
is a constant that does not depend on $m$.

Now if $N>2Q2^{4i_{0}(\alpha+1)}$, then we can choose an integer $m>i_{0}$ such that
\begin{equation*}
		\begin{split}
2Q2^{4(m-1)(\alpha+1)}<N\leq2Q2^{4m(\alpha+1)}
		\end{split}
	\end{equation*}
Therefore, at least half of $1,\cdots,N$ do not belong to $\bigcup_{i=i_{0}}^{m-1}P_{i}$ and
we have
\begin{equation*}
		\begin{split}
\sum^{N}_{n=1}\rho^{p}_{p}(W^{\alpha}_{g}(\varphi_{n}))\geq\frac{N}{2}2^{(m-1)p}\geq
N^{1+\frac{p}{4(\alpha+1)}}\left( \frac{3\Gamma(2\alpha+3)}{2^{\frac{12}{p}(\alpha+1)+6\alpha+5}}\right)^{\frac{p}{4(\alpha+1)}}
		\end{split}
	\end{equation*}

Finally, if $N\leq2Q2^{4i_{0}(\alpha+1)}$, then we have
\begin{equation*}
		\begin{split}
\sum^{N}_{n=1}\rho^{p}_{p}(W^{\alpha}_{g}(\varphi_{n}))\geq N2^{(m-1)p}\geq
N^{1+\frac{p}{4(\alpha+1)}}\left( \frac{3\Gamma(2\alpha+3)}{2^{\frac{8}{p}(\alpha+1)+6\alpha+8}}\right)^{\frac{p}{4(\alpha+1)}}
		\end{split}
	\end{equation*}
This completes the proof.
\end{proof}

\begin{thm}
Let $g\in L^{2}_{\alpha}(\mathbb{R_{+}})$ be a window function such that $\|g\|^{2}_{L^{2}_{\alpha}(\mathbb{R_{+}}\times \mathbb{R_{+}^{*}})}=1$. Suppose that
 $\|f\|^{2}_{L^{2}_{\alpha}(\mathbb{R_{+}}\times \mathbb{R_{+}^{*}})}=1$, then for $E\subset \mathbb{R_{+}}\times \mathbb{R_{+}^{*}}$ and $\eta\geq0$ such that
\begin{equation*}
		\begin{split}
 \int_{E}\int_{E}\mid W^{\alpha}_{g}(f)(k,s)\mid^{2}\emph{d}\nu_{\alpha}(k,s)\geq1-\eta
		\end{split}
	\end{equation*}
we have
\begin{equation*}
		\begin{split}
 \nu_{\alpha}(E)\geq1-\eta
		\end{split}
	\end{equation*}
\end{thm}
\begin{proof}
According to (18), we have
\begin{equation*}
		\begin{split}
 1-\eta\leq\int_{E}\int_{E}\mid W^{\alpha}_{g}(f)(k,s)\mid^{2}\emph{d}\nu_{\alpha}(k,s)\leq
 \|W^{\alpha}_{g}(f)\|_{L^{\infty}_{\alpha}(\mathbb{R_{+}}\times \mathbb{R_{+}^{*}})}\nu_{\alpha}(E)\leq\nu_{\alpha}(E)
		\end{split}
	\end{equation*}
\end{proof}
\begin{pro}
Let $g\in L^{2}_{\alpha}(\mathbb{R_{+}})$ be a window function such that $\|g\|^{2}_{L^{2}_{\alpha}(\mathbb{R_{+}}\times \mathbb{R_{+}^{*}})}=1$. Then for every function
$f\in L^{2}_{\alpha}(\mathbb{R_{+}})$ and $E\subset\mathbb{R_{+}}\times \mathbb{R_{+}^{*}}$ such that $\nu_{\alpha}(E)<1$,
then
\begin{equation*}
		\begin{split}
 \|\chi_{E^{c}}W^{\alpha}_{g}(f)\|_{L^{2}_{\alpha}(\mathbb{R_{+}}\times \mathbb{R_{+}^{*}})}\geq\sqrt{1-\nu_{\alpha}(E)}\|f\|_{L^{2}_{\alpha}(\mathbb{R_{+}})}
		\end{split}
	\end{equation*}
\end{pro}
\begin{proof}
By (18), we have
\begin{equation*}
		\begin{split}
 \|W^{\alpha}_{g}(f)\|^{2}_{L^{2}_{\alpha}(\mathbb{R_{+}}\times \mathbb{R_{+}^{*}})}&=\|\chi_{E}W^{\alpha}_{g}(f)+\chi_{E^{c}}W^{\alpha}_{g}(f)
 \|^{2}_{L^{2}_{\alpha}(\mathbb{R_{+}}\times \mathbb{R_{+}^{*}})}\\
 &\leq \|\chi_{E}W^{\alpha}_{g}(f)\|^{2}_{L^{2}_{\alpha}(\mathbb{R_{+}}\times \mathbb{R_{+}^{*}})}+ \|\chi_{E^{c}}W^{\alpha}_{g}(f)\|^{2}_{L^{2}_{\alpha}(\mathbb{R_{+}}\times \mathbb{R_{+}^{*}})}\\
&\leq\nu_{\alpha}(E)\|W^{\alpha}_{g}(f)\|^{2}_{L^{\infty}_{\alpha}(\mathbb{R_{+}}\times \mathbb{R_{+}^{*}})}+ \|\chi_{E^{c}}W^{\alpha}_{g}(f)\|^{2}_{L^{2}_{\alpha}(\mathbb{R_{+}}\times \mathbb{R_{+}^{*}})}\\
&\leq\nu_{\alpha}(E)\|g\|^{2}_{L^{2}_{\alpha}(\mathbb{R_{+}})}\|f\|^{2}_{L^{2}_{\alpha}(\mathbb{R_{+}})}+ \|\chi_{E^{c}}W^{\alpha}_{g}(f)\|^{2}_{L^{2}_{\alpha}(\mathbb{R_{+}}\times \mathbb{R_{+}^{*}})}\\
		\end{split}
	\end{equation*}
Hence
\begin{equation*}
		\begin{split}
 \|\chi_{E^{c}}W^{\alpha}_{g}(f)\|_{L^{2}_{\alpha}(\mathbb{R_{+}}\times \mathbb{R_{+}^{*}})}\geq\sqrt{1-\nu_{\alpha}(E)}\|f\|_{L^{2}_{\alpha}(\mathbb{R_{+}})}
		\end{split}
	\end{equation*}
\end{proof}
\begin{thm} (Local uncertainty inequality for the windowed Hankel
transform)
Let $x$ be a real number such that $0<x<\alpha+1$, $t\geq0$, then for every non-zero function
$f\in L^{2}_{\alpha}(\mathbb{R_{+}})$ and for every measurable subset $E\subset\mathbb{R_{+}}\times \mathbb{R_{+}^{*}}$ such that
$0<\nu_{\alpha}(E)<+\infty$, we have
\begin{equation}
		\begin{split}
		\|\chi_{E}W^{\alpha}_{g}(f)\|_{L^{2}_{\alpha}(\mathbb{R_{+}}\times \mathbb{R_{+}^{*}})}&\leq
\left(1+\frac{x^{2}}{(\nu_{\alpha}(E))^{1/2}(\alpha+1-x)^{2}}\right)
\left(\frac{x^{2}2^{\alpha+1}\Gamma(\alpha+1)}{\nu_{\alpha}(E)(\alpha+1-x)} \right)^{-x/(\alpha+1)}
\|t^{x}f\|_{L^{2}_{\alpha}(\mathbb{R_{+}})}\|t^{x}g\|_{L^{2}_{\alpha}(\mathbb{R_{+}})}
		\end{split}
	\end{equation}
\end{thm}
\begin{proof}
For every $a>0$, according to Proposition 3, we have
\begin{equation*}
		\begin{split}
\|\chi_{E}W^{\alpha}_{g}(f)\|_{L^{2}_{\alpha}(\mathbb{R_{+}}\times \mathbb{R_{+}^{*}})}
\leq (\nu_{\alpha}(E))^{1/2}\|\chi_{[0,a)}f\|_{L^{2}_{\alpha}(\mathbb{R_{+}})}\|\chi_{[0,a)}g\|_{L^{2}_{\alpha}
(\mathbb{R_{+}})}+
\|W^{\alpha}_{g}(\chi_{[a,+\infty)}f)\|_{L^{2}_{\alpha}(\mathbb{R_{+}}\times \mathbb{R_{+}^{*}})}
		\end{split}
	\end{equation*}
However, by H$\ddot{o}$lder's inequality, (3) and (6), we have
 \begin{equation*}
		\begin{split}
 \|\chi_{[0,a)}f\|_{L^{2}_{\alpha}(\mathbb{R_{+}})}
 &\leq\|t^{x}f\|_{L^{2}_{\alpha}}\|\chi_{[0,a)}t^{-x}\|_{L^{2}_{\alpha}(\mathbb{R_{+}})}\\
 &=\|t^{x}f\|_{L^{2}_{\alpha}(\mathbb{R_{+}})}\frac{a^{\alpha+1-x}}{(2^{\alpha+1}\Gamma(\alpha+1)(\alpha+1-x))^{1/2}}
		\end{split}
	\end{equation*}
By (19), we have
\begin{equation*}
		\begin{split}
 \|W^{\alpha}_{g}(\chi_{[a,+\infty)}f)\|_{L^{2}_{\alpha}(\mathbb{R_{+}}\times \mathbb{R_{+}^{*}})}&=\|\chi_{[a,+\infty)}f\|_{L^{2}_{\alpha}(\mathbb{R_{+}})}
 \|\chi_{[a,+\infty)}g\|_{L^{2}_{\alpha}(\mathbb{R_{+}})}\\
 &\leq\|t^{x}f\|_{L^{2}_{\alpha}(\mathbb{R_{+}})}\|\chi_{[a,+\infty)}t^{-x}\|_{L^{\infty}_{\alpha}(\mathbb{R_{+}})}
 \|t^{x}g\|_{L^{2}_{\alpha}(\mathbb{R_{+}})}\|\chi_{[a,+\infty)}t^{-x}\|_{L^{\infty}_{\alpha}(\mathbb{R_{+}})}\\
 &=a^{-2x}\|t^{x}f\|_{L^{2}_{\alpha}(\mathbb{R_{+}})}\|t^{x}g\|_{L^{2}_{\alpha}(\mathbb{R_{+}})}
		\end{split}
	\end{equation*}
Hence
\begin{equation}
		\begin{split}
		\|\chi_{E}W^{\alpha}_{g}(f)\|_{L^{2}_{\alpha}(\mathbb{R_{+}}\times \mathbb{R_{+}^{*}})}
&\leq \left(a^{-2x}+\frac{(\nu_{\alpha}(E))^{1/2}}{2^{\alpha+1}\Gamma(\alpha+1)(\alpha+1-x)} a^{2\alpha+2-2x}
\right)\|t^{x}f\|_{L^{2}_{\alpha}(\mathbb{R_{+}})}\|t^{x}g\|_{L^{2}_{\alpha}(\mathbb{R_{+}})}
		\end{split}
	\end{equation}
in particular inequality (32) holds for
\begin{equation*}
		\begin{split}
		a_{0}=\left(\frac{x^{2}2^{\alpha+1}\Gamma(\alpha+1)}{\nu_{\alpha}(E)(\alpha+1-x)} \right)^{1/2(\alpha+1)}
		\end{split}
	\end{equation*}
Hence
\begin{equation*}
		\begin{split}
		\|\chi_{E}W^{\alpha}_{g}(f)\|_{L^{2}_{\alpha}(\mathbb{R_{+}}\times \mathbb{R_{+}^{*}})}&\leq
\left(1+\frac{x^{2}}{(\nu_{\alpha}(E))^{1/2}(\alpha+1-x)^{2}}\right)
\left(\frac{x^{2}2^{\alpha+1}\Gamma(\alpha+1)}{\nu_{\alpha}(E)(\alpha+1-x)} \right)^{-x/(\alpha+1)}
\|t^{x}f\|_{L^{2}_{\alpha}(\mathbb{R_{+}})}\|t^{x}g\|_{L^{2}_{\alpha}(\mathbb{R_{+}})}
		\end{split}
	\end{equation*}
\end{proof}

According to the following logarithmic uncertainty principle for the Hankel transform \cite{[3]},
we obtain the logarithmic uncertainty principle for the windowed Hankel transform.

\begin{pro}
For every $f\in L^{2}_{\alpha}(\mathbb{R_{+}})$, the following inequality holds:
\begin{equation}
		\begin{split} \int^{+\infty}_{0}\ln(t)|f(t)|^{2}\emph{d}\gamma_{\alpha}(t)+\int^{+\infty}_{0}\ln(\lambda)|\mathcal{H}_{\alpha}(f)(\lambda)|^{2}
\emph{d}\gamma_{\alpha}(\lambda)\geq
\left( \psi\left( \frac{\alpha+1}{2}\right)+\ln(2)\right)\int^{+\infty}_{0}|f(t)|^{2}\emph{d}\gamma_{\alpha}(t)
		\end{split}
	\end{equation}
where $\psi$ denotes the logarithmic derivative of the Euler function $\Gamma$ [9,11].
\end{pro}
Now we arrive at the following important result.

\begin{thm} (Logarithmic uncertainty principle for the windowed Hankel transform)
For every $g\in L^{2}_{\alpha}(\mathbb{R_{+}})$ and $f\in L^{1}_{\alpha}(\mathbb{R_{+}})\bigcap L^{2}_{\alpha}(\mathbb{R_{+}})$ we have the following inequality:
\begin{equation}
		\begin{split} \int^{+\infty}_{0}\int^{+\infty}_{0}\ln(k)|W^{\alpha}_{g}(f)(k,s)|^{2}\emph{d}\nu_{\alpha}(k,s)&+
\int^{+\infty}_{0}\int^{+\infty}_{0}\ln(s)|W^{\alpha}_{f}(g)(k,s)|^{2}\emph{d}\nu_{\alpha}(k,s)\\
&\geq\left( \psi\left( \frac{\alpha+1}{2}\right)+\ln(2)\right)\|f\|^{2}_{L^{2}_{\alpha}(\mathbb{R_{+}}\times \mathbb{R_{+}^{*}})}\|g\|^{2}_{L^{2}_{\alpha}(\mathbb{R_{+}}\times \mathbb{R_{+}^{*}})}
		\end{split}
	\end{equation}
where $\psi$ denotes the logarithmic derivative of the Euler function $\Gamma$.
\end{thm}
\begin{proof}
We can replace the logarithmic uncertainty principle for the Hankel transform to the function $W^{\alpha}_{g}(f)(k,s)$, then we have
\begin{equation*}
		\begin{split}
		\int^{+\infty}_{0}\ln(k)|W^{\alpha}_{g}(f)(k,s)|^{2}\emph{d}\gamma_{\alpha}(k)&+
\int^{+\infty}_{0}\ln(\lambda)|\mathcal{H}_{\alpha}(W^{\alpha}_{g}(f)(\cdot,s))(\lambda)|^{2}\emph{d}\gamma_{\alpha}(\lambda)\\
&\geq\left( \psi\left( \frac{\alpha+1}{2}\right)+\ln(2)\right)\int^{+\infty}_{0}|W^{\alpha}_{g}(f)(k,s)|^{2}\emph{d}\gamma_{\alpha}(k)
		\end{split}
	\end{equation*}
Integrating both sides with respect to $s$, we have
\begin{equation}
		\begin{split}
		\int^{+\infty}_{0}\int^{+\infty}_{0}\ln(k)|W^{\alpha}_{g}(f)(k,s)|^{2}\emph{d}\nu_{\alpha}(k,s)&+
\int^{+\infty}_{0}\int^{+\infty}_{0}\ln(\lambda)|\mathcal{H}_{\alpha}(W^{\alpha}_{g}(f)(\cdot,s))
(\lambda)|^{2}\emph{d}\nu_{\alpha}
(\lambda,s)\\
&\geq\left( \psi\left( \frac{\alpha+1}{2}\right)+\ln(2)\right)\int^{+\infty}_{0}\int^{+\infty}_{0}
|W^{\alpha}_{g}(f)(k,s)|^{2}\emph{d}\nu_{\alpha}(k,s)
		\end{split}
	\end{equation}

By (12) and (17), we obtain
\begin{equation}
		\begin{split}
\int^{+\infty}_{0}\int^{+\infty}_{0}\ln(\lambda)|\mathcal{H}_{\alpha}(W^{\alpha}_{g}(f)(\cdot,s))
(\lambda)|^{2}\emph{d}\nu_{\alpha}
(\lambda,s)
=\|g\|^{2}_{L^{2}_{\alpha}(\mathbb{R_{+}})}\int^{+\infty}_{0}\ln(\lambda)
|\mathcal{H}_{\alpha}(f)(\lambda)|^{2}\emph{d}\gamma_{\alpha}(\lambda)
		\end{split}
	\end{equation}
Hence, by (20) we have
\begin{equation}
		\begin{split} \int^{+\infty}_{0}\int^{+\infty}_{0}\ln(k)|W^{\alpha}_{g}(f)(k,s)|^{2}\emph{d}\nu_{\alpha}(k,s)&+
\|g\|^{2}_{L^{2}_{\alpha}(\mathbb{R_{+}})}\int^{+\infty}_{0}\ln(\lambda)|\mathcal{H}_{\alpha}(f)(\lambda)|^{2}
\emph{d}\gamma_{\alpha}
(\lambda)\\
&\geq\left( \psi\left( \frac{\alpha+1}{2}\right)+\ln(2)\right)\|f\|^{2}_{L^{2}_{\alpha}(\mathbb{R_{+}})}\|g\|^{2}_{L^{2}_{\alpha}
(\mathbb{R_{+}})}
		\end{split}
	\end{equation}
On the other hand, by (5) and (12) we have
\begin{equation*}
		\begin{split} \int^{+\infty}_{0}\int^{+\infty}_{0}\ln(s)|W^{\alpha}_{f}(g)(k,s)|^{2}\emph{d}\nu_{\alpha}(k,s)&=
\int^{+\infty}_{0}\ln(s)\left[\int^{+\infty}_{0}|\mathcal{H}_{\alpha}[W^{\alpha}_{f}(g)(\cdot,s)](\omega)|^{2}
\emph{d}\gamma_{\alpha}(\omega)\right]\emph{d}\gamma_{\alpha}(s)\\
&=\int^{+\infty}_{0}\ln(s)\left[\int^{+\infty}_{0}|\mathcal{H}_{\alpha}(g)(\omega)|^{2}
\tau^{\alpha}_{s}|\mathcal{H}_{\alpha}(f)(\omega)|^{2}
\emph{d}\gamma_{\alpha}(\omega)\right]\emph{d}\gamma_{\alpha}(s)\\
&=\int^{+\infty}_{0}|\mathcal{H}_{\alpha}(g)(\omega)|^{2}\left[\int^{+\infty}_{0}\ln(s)
\tau^{\alpha}_{\omega}|\mathcal{H}_{\alpha}(f)(s)|^{2}
\emph{d}\gamma_{\alpha}(s)\right]\emph{d}\gamma_{\alpha}(\omega)\\
		\end{split}
	\end{equation*}
According to (10), we have
\begin{equation*}
		\begin{split}
\int^{+\infty}_{0}\ln(s)\tau^{\alpha}_{\omega}|\mathcal{H}_{\alpha}(f)|^{2}(s)
\emph{d}\gamma_{\alpha}(s)=\int^{+\infty}_{0}
\ln(s)|\mathcal{H}_{\alpha}(f)|^{2}(s)
\emph{d}\gamma_{\alpha}(s)\\
		\end{split}
	\end{equation*}
Hence
\begin{equation}
		\begin{split}
\int^{+\infty}_{0}\int^{+\infty}_{0}\ln(s)|W^{\alpha}_{f}(g)(k,s)|^{2}\emph{d}\nu_{\alpha}(k,s)&
=\int^{+\infty}_{0}|\mathcal{H}_{\alpha}(g)(\omega)|^{2}\emph{d}\gamma_{\alpha}(\omega)
\int^{+\infty}_{0}\ln(s)|\mathcal{H}_{\alpha}(f)|^{2}(s)\emph{d}\gamma_{\alpha}(s)\\
&=\|g\|^{2}_{L^{2}_{\alpha}(\mathbb{R_{+}})}\int^{+\infty}_{0}\ln(s)
|\mathcal{H}_{\alpha}(f)(s)|^{2}\emph{d}\gamma_{\alpha}(s)
		\end{split}
	\end{equation}
Which completes the proof.
\end{proof}
In the paper \cite{[16]},  Soltani gave explicitly the constant $b$ in the case $c\geq1$ and $d\geq1$,
more precisely he established the following Heisenberg-type inequalities for the Hankel transform.
\begin{pro}
Assume $c\geq1$ and $d\geq1$, then for every function $f\in L^{2}_{\alpha}(\mathbb{R_{+}})$, we have
\begin{equation}
		\begin{split}
\|t^{d}f\|^{c/(c+d)}_{L^{2}_{\alpha}(\mathbb{R_{+}})}
\|\lambda^{c}\mathcal{H}_{\alpha}(f)\|^{d/(c+d)}_{L^{2}_{\alpha}(\mathbb{R_{+}})}\geq (\alpha+1)^{cd/(c+d)}\|f\|_{L^{2}_{\alpha}(\mathbb{R_{+}})}
		\end{split}
	\end{equation}
with equality if and only if $c=d=1$ and $f(t)=pe^{-qt^{2}/2}$ for some $p\in\mathbb{C}$ and $q>0$.
\end{pro}
In the paper \cite{[8]}, the authors gave the following Heisenberg-type uncertainty inequality for the windowed Hankel
transform of magnitude $c\geq 1$.
\begin{pro}
Let $c\geq1$, then for every function $f, g\in L^{2}_{\alpha}(\mathbb{R_{+}})$,
\begin{equation}
		\begin{split}
\|k^{c}W^{\alpha}_{g}(f)\|_{L^{2}_{\alpha}(\mathbb{R_{+}}\times \mathbb{R_{+}^{*}})}\|s^{c}W^{\alpha}_{g}(f)\|_{L^{2}_{\alpha}(\mathbb{R_{+}}\times \mathbb{R_{+}^{*}})}\geq
 (\alpha+1)^{c}\|g\|^{2}_{L^{2}_{\alpha}(\mathbb{R_{+}})}\|f\|^{2}_{L^{2}_{\alpha}(\mathbb{R_{+}})}
		\end{split}
	\end{equation}
\end{pro}
When $c=1$, the uncertainty inequality become the following uncertainty inequality:
\begin{equation}
		\begin{split}
\|kW^{\alpha}_{g}(f)\|_{L^{2}_{\alpha}(\mathbb{R_{+}}\times \mathbb{R_{+}^{*}})}\|sW^{\alpha}_{g}(f)\|_{L^{2}_{\alpha}(\mathbb{R_{+}}\times \mathbb{R_{+}^{*}})}\geq
 (\alpha+1)\|g\|^{2}_{L^{2}_{\alpha}(\mathbb{R_{+}})}\|f\|^{2}_{L^{2}_{\alpha}(\mathbb{R_{+}})}
		\end{split}
	\end{equation}

Now we investigate the Heisenberg-type uncertainty inequality for the windowed Hankel
transform.
\begin{thm} (Heisenberg-type uncertainty inequality for the windowed Hankel
transform)
Assume $c,d\geq1$ and let $g \in L^{2}_{\alpha}(\mathbb{R_{+}})$ be a non-zero window function. For every non-zero function $f\in L^{2}_{\alpha}(\mathbb{R_{+}})$,
then there exits a constant $b=(\alpha,c,d)>0$,
such that
\begin{equation}
		\begin{split}
\|k^{c}W^{\alpha}_{g}(f)\|^{d/(c+d)}_{L^{2}_{\alpha}(\mathbb{R_{+}}\times \mathbb{R_{+}^{*}})}\|s^{d}W^{\alpha}_{g}(f)\|^{c/(c+d)}_{L^{2}_{\alpha}(\mathbb{R_{+}}\times \mathbb{R_{+}^{*}})}\geq
(\alpha+1)^{cd/(c+d)}\|g\|_{L^{2}_{\alpha}(\mathbb{R_{+}})}\|f\|_{L^{2}_{\alpha}(\mathbb{R_{+}})}
		\end{split}
	\end{equation}
\end{thm}
\begin{proof}
Let $f, g \in L^{2}_{\alpha}(\mathbb{R_{+}})$ be tow non-zero functions, such that
\begin{equation*}
		\begin{split}
\|k^{c}W^{\alpha}_{g}(f)\|_{L^{2}_{\alpha}(\mathbb{R_{+}}\times \mathbb{R_{+}^{*}})}; \|s^{d}W^{\alpha}_{g}(f)\|_{L^{2}_{\alpha}(\mathbb{R_{+}}\times \mathbb{R_{+}^{*}})}<\infty
		\end{split}
	\end{equation*}
Then, for $c>1$, we have
\begin{equation*}
		\begin{split}
\|k^{c}W^{\alpha}_{g}(f)\|^{1/c}_{L^{2}_{\alpha}(\mathbb{R_{+}}\times \mathbb{R_{+}^{*}})}\|W^{\alpha}_{g}(f)\|^{1/c'}_{L^{2}_{\alpha}(\mathbb{R_{+}}\times \mathbb{R_{+}^{*}})}
=\|k^{2}|W^{\alpha}_{g}(f)|^{2/c}\|^{1/2}_{L^{c}_{\alpha}(\mathbb{R_{+}}\times \mathbb{R_{+}^{*}})}\||W^{\alpha}_{g}(f)|^{2/c'}\|^{1/2}_{L^{c'}_{\alpha}(\mathbb{R_{+}}\times \mathbb{R_{+}^{*}})}
		\end{split}
	\end{equation*}
where $c'$ is defined as usual by $1/c+1/c'=1$.

By H$\ddot{o}$lder's inequality, we obtain
\begin{equation*}
		\begin{split}
\|k^{c}W^{\alpha}_{g}(f)\|^{1/c}_{L^{2}_{\alpha}(\mathbb{R_{+}}\times \mathbb{R_{+}^{*}})}\|W^{\alpha}_{g}(f)\|^{1/c'}_{L^{2}_{\alpha}(\mathbb{R_{+}}\times \mathbb{R_{+}^{*}})}
\geq\|kW^{\alpha}_{g}(f)\|_{L^{2}_{\alpha}(\mathbb{R_{+}}\times \mathbb{R_{+}^{*}})}
		\end{split}
	\end{equation*}
Thus, for all $c\geq 1$, we have
\begin{equation}
		\begin{split}
\|k^{c}W^{\alpha}_{g}(f)\|^{1/c}_{L^{2}_{\alpha}(\mathbb{R_{+}}\times \mathbb{R_{+}^{*}})}\geq
\frac{\|kW^{\alpha}_{g}(f)\|_{L^{2}_{\alpha}(\mathbb{R_{+}}\times \mathbb{R_{+}^{*}})}}{\|W^{\alpha}_{g}(f)\|^{1-1/c}_{L^{2}_{\alpha}(\mathbb{R_{+}}\times \mathbb{R_{+}^{*}})}}
		\end{split}
	\end{equation}
with equality if $c = 1$.

In the same way, we obtain for $d\geq1$,
\begin{equation}
		\begin{split}
\|s^{d}W^{\alpha}_{g}(f)\|^{1/d}_{L^{2}_{\alpha}(\mathbb{R_{+}}\times \mathbb{R_{+}^{*}})}\geq
\frac{\|sW^{\alpha}_{g}(f)\|_{L^{2}_{\alpha}(\mathbb{R_{+}}\times \mathbb{R_{+}^{*}})}}{\|W^{\alpha}_{g}(f)\|^{1-1/d}_{L^{2}_{\alpha}(\mathbb{R_{+}}\times \mathbb{R_{+}^{*}})}}
		\end{split}
	\end{equation}
with equality if $d = 1$.

According to (43) and (44), for all $c, d\geq1$ we obtain
\begin{equation}
		\begin{split}
\|k^{c}W^{\alpha}_{g}(f)\|^{d/(c+d)}_{L^{2}_{\alpha}(\mathbb{R_{+}}\times \mathbb{R_{+}^{*}})}\|s^{d}W^{\alpha}_{g}(f)\|^{c/(c+d)}_{L^{2}_{\alpha}(\mathbb{R_{+}}\times \mathbb{R_{+}^{*}})}\geq
\left(\frac{\|kW^{\alpha}_{g}(f)\|_{L^{2}_{\alpha}(\mathbb{R_{+}}\times \mathbb{R_{+}^{*}})}\|sW^{\alpha}_{g}(f)\|_{L^{2}_{\alpha}(\mathbb{R_{+}}\times \mathbb{R_{+}^{*}})}}
{\|W^{\alpha}_{g}(f)\|^{2-1/c-1/d}_{L^{2}_{\alpha}(\mathbb{R_{+}}\times \mathbb{R_{+}^{*}})}}\right)^{cd/(c+d)}
		\end{split}
	\end{equation}
with equality if $c= d = 1$.

By (19) and (41), we obtain the desired result.
\end{proof}

\medskip

\noindent  {\bf Acknowledge}

This work is supported by the National Natural Science Foundation of China (No. 61671063).


\end{document}